\documentclass[10pt,a4paper]{article}
\usepackage[active]{srcltx}

\usepackage[final]{optional}
\usepackage[british,american]{babel}

\usepackage{amsfonts, amsmath, wasysym}

\ifx\pdfoutput\undefined
\usepackage{graphicx}
\else
\usepackage[pdftex]{graphicx}
\usepackage{epstopdf}
\fi

\usepackage{amssymb,amsthm,
paralist
}

\usepackage{
latexsym,
nicefrac,
}

\usepackage[usenames]{color}
\usepackage{url}

\definecolor{darkgreen}{rgb}{0,0.4,0}
\definecolor{darkred}{rgb}{0.5,0,0}
\definecolor{violet}{rgb}{0.3,0,0.3}
\usepackage[colorlinks, urlcolor=violet,
citecolor=darkgreen, linkcolor=darkred
]{hyperref}
\usepackage{esint}
\usepackage{multirow}
\usepackage{bibgerm}
\usepackage[normalem]{ulem}

\theoremstyle{plain}
\newtheorem{lemma}{Lemma}[section]
\newtheorem{thm}[lemma]{Theorem}
\newtheorem{prop}[lemma]{Proposition}
\newtheorem{cor}[lemma]{Corollary}

\theoremstyle{definition}
\newtheorem{defn}[lemma]{Definition}

\newtheorem{rmk}[lemma]{Remark}

\def\blbox{\quad \vrule height7.5pt width4.17pt depth0pt}
\newcommand{\cmt}[1]{\opt{draft}{\textcolor[rgb]{0.5,0,0}{
$\LHD$ #1 $\RHD$\marginpar{\blbox}}}}

\numberwithin{equation}{section}

\newcommand{\m}{\ensuremath{{\cal M}}}

\newcommand{\al}{\alpha}

\newcommand{\N}{\ensuremath{{\mathbb N}}}

\newcommand{\C}{\ensuremath{{\mathbb C}}}

\DeclareMathOperator{\Vol}{vol}

\def\blbox{\quad \vrule height7.5pt width4.17pt depth0pt}

\newcommand{\beq}{\begin{equation}}
\newcommand{\eeq}{\end{equation}}
\newcommand{\beqa}{\begin{equation}\begin{aligned}}
\newcommand{\eeqa}{\end{aligned}\end{equation}}
\newcommand{\brmk}{\begin{rmk}}
\newcommand{\ermk}{\end{rmk}}
\newcommand{\partref}[1]{\hbox{(\csname @roman\endcsname{\ref{#1}})}}

\newcommand{\GK}{{\mathrm{K}}}

\newcommand{\Cts}{{\mathrm{C}}}

\newcommand*{\ee}{\mathrm{e}}

\newcommand*\im{\mathop{\mathrm{im}}\nolimits}
\newcommand*\width{\mathop{\mathrm{w}}\nolimits}

\newcommand*\dist{\mathop{\mathrm{dist}}\nolimits}

\newcommand*\Mf{\mathcal{M}}
\newcommand*\Disc{\mathbb{D}}

\def\L{\mathrm{L}}

\newcommand*\gBall{\mathcal{B}}
\newcommand*\sph{\mathbb{S}}

\newcommand*\ddt{\frac{\mathrm{d}}{\mathrm{d}t}}
\newcommand*\pddt{\frac{\partial}{\partial t}}

\newcommand*\df{\mathrm{d}}
\newcommand*\dmu{\,\mathrm{d}\mu}

\newcommand*\ds{\mathrm{d}s}

\newcommand*\dz{\mathrm{d}z}
\newcommand*\Ball{\mathbb{B}}

\title{{\sc
ricci flows with bursts of \\ unbounded curvature
}
\\
\cmt{DRAFT with comments}
}
\author{Gregor Giesen and Peter M. Topping}
\date{\today}

\begin{document}
\selectlanguage{british}

\maketitle

\begin{abstract}
\noindent 
Given a completely arbitrary surface, whether or not it has bounded curvature, or even whether or not it is complete, there exists an instantaneously complete Ricci flow evolution of that surface that exists for a specific amount of time \cite{GT11}. In the case that the underlying Riemann surface supports a hyperbolic metric, this Ricci flow always exists for all time and converges (after scaling by a factor $\frac{1}{2t}$) to this hyperbolic metric \cite{GT11}, i.e. our Ricci flow \emph{geometrises} the surface. In this paper we show that there exist complete, bounded curvature initial metrics, including those conformal to a hyperbolic metric, which have subsequent Ricci flows 
developing unbounded curvature at certain intermediate times. In particular, when coupled with the uniqueness from \cite{Top13}, we find that \emph{any} complete Ricci flow  starting with such initial metrics \emph{must} develop unbounded curvature over some intermediate time interval, but that nevertheless, 
the curvature \emph{must} later become bounded and the flow must achieve geometrisation as $t\to\infty$, even though
there are other conformal deformations to hyperbolic metrics that do not involve unbounded curvature.

Another consequence of our constructions is that while our Ricci flow from \cite{GT11} must agree initially with the classical  flow of Hamilton and Shi in the special case that the initial surface is complete and of bounded curvature, by uniqueness, it is now clear that our flow lasts for a longer time interval in general, with Shi's flow stopping when the curvature blows up, but our flow continuing strictly beyond in these situations. 

All our constructions of unbounded curvature developing and then disappearing are in two dimensions. Generalisations to higher dimensions are then immediate.
%
\end{abstract}

\section{Introduction}

Hamilton \cite{Ham82} and Shi \cite{Shi89} proved that given a complete Riemannian manifold $(\m,g_0)$ with bounded curvature, there exists a complete Ricci flow  $g(t)$ on $\m$ for a short time, with $g(0)=g_0$ (see \cite{Top06} for an introduction to this topic).
The curvature of this Ricci flow is initially bounded, and the flow can be extended until such time that the curvature becomes unbounded.

Ricci flows with possibly \emph{unbounded} curvature in their initial condition and/or during the flow itself, were studied by the second author in \cite{Top10} in the special case of surfaces, and in \cite{GT11} we proved that one can always find an \emph{instantaneously complete} Ricci flow
starting at a completely arbitrary initial surface, whether of unbounded curvature or not, or indeed whether complete or not, which exists for a specific amount of time, and in \cite{Top13} this solution was shown to be unique. More precisely, we proved:


\begin{thm}[{Part of \cite[Theorem 1.3]{GT11} and \cite[Theorem 1.1]{Top13}}]
\label{thm:2d-existence}
Let $\bigl(\Mf^2,g_0\bigr)$ be a smooth Riemannian surface which 
could have unbounded curvature or be incomplete.
Depending on the conformal class,
we define $T\in(0,\infty]$ by
\begin{equation}
\label{eq:max-time}
T := \left\{\begin{array}{cl} \displaystyle
\frac{\Vol_{g_0}\Mf}{4\pi\mathop\chi(\Mf)} &
\text{if $\bigl(\Mf,g_0\bigr)
\stackrel{\mathrm{conf}}{=}\bigl(\sph^2,g_\sph\bigr)$ or
$\bigl(\mathbb RP^2,g_\sph\bigr)$ or
$\bigl(\mathbb C,|\dz|^2\bigr)$\footnotemark,}\\
\infty & \text{otherwise.}
\end{array}\right.
\end{equation}
\footnotetext{
Note that in the latter case,
also $T=\infty$ if $\Vol_{g_0}\mathbb C=\infty$.}
Then there exists a smooth Ricci flow $\bigl(g(t)\bigr)_{t\in[0,T)}$
such that
\begin{compactenum}[(i)]\medskip
\item $g(0)=g_0$;\smallskip
\item $\bigl(g(t)\bigr)_{t\in[0,T)}$ is instantaneously complete (i.e. $g(t)$ is complete for all $t\in (0,T)$);
\item $\bigl(g(t)\bigr)_{t\in[0,T)}$ is maximally stretched (see Remark \ref{max_stretch}),
\end{compactenum}\medskip
and this flow is unique in the sense that if
$\bigl(g_2(t)\bigr)_{t\in[0,T_2)}$ is
any other Ricci flow on $\Mf$ satisfying (i) and (ii), then
$T_2\leq T$ and $g_2(t)=g(t)$ for all $t\in[0,T_2)$.

\noindent
If $T<\infty$, then we have
\begin{equation}
\label{eq:2d-existince-volume}
\Vol_{g(t)}\Mf = 4\pi\mathop\chi(\Mf)\, (T-t) \longrightarrow 0
\quad\text{ as } t\nearrow T,
\end{equation}
and in particular, $T$ is the maximal existence time.
\end{thm}

It has been understood since the work of Hamilton and Chow \cite{Ham88, Cho91} that the Ricci flow on \emph{closed} surfaces has excellent geometrisation properties in the sense that after appropriate rescaling, the flow converges to a metric of constant curvature. 
It is then natural to ask to which extent this geometrisation occurs in the case that the underlying surface is noncompact.
In \cite[Theorem 1.3]{GT11} we also proved that geometrisation does indeed occur in the hyperbolic case.

\begin{thm}[{Special case of part of \cite[Theorem 1.3]{GT11}}]
\label{geom_thm}
Let $\bigl(\Mf^2, g_0\bigr)$ be any surface, possibly incomplete or of unbounded curvature, that is conformal to a complete hyperbolic metric $H$. Then the Ricci flow from Theorem \ref{thm:2d-existence}, which exists for all time $t\geq 0$, must converge in the sense that 
$$\frac{1}{2t}\,g(t)\longrightarrow H\quad\text{ smoothly locally as }t\to\infty.$$
Moreover, if there exists $M<\infty$ such that $g_0\leq MH$, then we have global convergence
$$\frac{1}{2t}\,g(t)\longrightarrow H\quad\text{ in }\Cts^k(\Mf, H)\text{ as }t\to\infty,$$
for any $k\in\N$.
\end{thm}

In this paper we prove that although the Ricci flow achieves geometrisation in this particularly simple form, the route it takes to get to the constant curvature metric is necessarily more complicated than one might initially expect.
Indeed, the following theorem finds smooth, complete initial metrics with bounded curvature, whose subsequent Ricci flows converge uniformly in $\Cts^k$ after rescaling to hyperbolic metrics, but for which the Ricci flow develops a burst of \emph{unbounded} curvature on some intermediate time interval.
Coupled with the uniqueness of Ricci flows proved in \cite{Top13}, this shows that Ricci flow has no choice but to develop this unbounded curvature, despite its nice initial and final behaviour, even though alternative deformations to constant curvature without developing unbounded curvature will exist.

\begin{thm}
\label{thm:ex-rf-curv-bursts}
There exist a complete, conformal, immortal Ricci flow
$\bigl(g(t)\bigr)_{t\in[0,\infty)}$ on the disc $\mathbb D$ arising from 
Theorem \ref{thm:2d-existence}, and a time
$t_1\in[1,3)$ such that
\beq
\label{mainclaim1}
\sup_\Mf \left|\GK_{g(t)}\right| \begin{cases}
<\infty & \text{for all $t\in[0,1)$} \\
=\infty & \text{for all $t\in\left(t_1,t_1+\nicefrac1{100}\right)$} \\
<\infty & \text{for all $t\in[4,\infty)$,}
\end{cases}
\eeq
and so that 
$$\frac1{2t}g(t)\longrightarrow H\qquad\text{ in }
\Cts^k(\mathbb D, H)\text{ as }t\to\infty,$$
where $H$ is the Poincar\'e metric on $\mathbb D$ and $k$ is any natural number, and in particular so that 
$$\GK_{\frac1{2t}g(t)}\longrightarrow -1\qquad\text{ uniformly as }t\to\infty.$$

Moreover, there exist a complete, conformal, immortal Ricci flow
$\bigl(g(t)\bigr)_{t\in[0,\infty)}$ on $\mathbb C$ arising from 
Theorem \ref{thm:2d-existence}, and a possibly different time
$t_1\in[1,3)$ such that \eqref{mainclaim1}  holds again.
\end{thm}

Note that by work of Chen \cite{Che09}, any complete Ricci flow on a surface has $\GK_{g(t)}\geq -\frac{1}{2t}$, so when the curvature is unbounded for later times, it is unbounded from above.

We pick our underlying Riemann surfaces to be $\mathbb D$ and $\mathbb C$ here to have examples in both the hyperbolic and parabolic cases. It is not hard to generalise to other underlying Riemann surfaces, at the cost of increased technicality. 

Our next theorem also has bearing on how the Ricci flow achieves its geometrisation, as we will explain after stating the result.

\begin{thm}
\label{thm:ex-rf-delayed-unbounded-curv}
On any noncompact underlying Riemann surface, there exists a complete, conformal, immortal Ricci flow $\bigl(g(t)\bigr)_{t\in[0,\infty)}$  arising from 
Theorem \ref{thm:2d-existence} such that
\[ \sup_\Mf \left|\GK_{g(t)}\right| \begin{cases}
<\infty & \text{for all $t\in[0,1)$} \\
=\infty & \text{for all $t\in[3,\infty)$.}
\end{cases} \]
\end{thm}

In particular, we could consider such an example $g(t)$ on the disc $\mathbb D$, on which the Poincar\'e metric $H$ lives. Since $H$ is hyperbolic, we know from Theorem \ref{geom_thm} that 
$\frac1{2t}g(t)$ converges to $H$ smoothly locally, and in particular that the curvature of $g(t)$ converges to \emph{zero} smoothly locally as $t\to\infty$. On the other hand our Theorem \ref{thm:ex-rf-delayed-unbounded-curv} claims that the curvature remains unbounded for all $t\geq 3$. 
We deduce that although the unbounded curvature necessarily forms, the Ricci flow organises itself in order to push the regions of large curvature out to spatial infinity as $t\to\infty$.

In higher dimensions, a theorem of the generality of Theorem \ref{thm:2d-existence} cannot be true. However, one can hope to prove the existence of Ricci flows starting with unbounded-curvature manifolds with certain positivity-of-curvature conditions, and   Cabezas-Rivas and Wilking \cite{CW11} have done this for positive complex sectional curvature.
The Ricci flows in our second result, Theorem \ref{thm:ex-rf-delayed-unbounded-curv} have similar properties to 
four-dimensional Ricci flows constructed by 
the same authors \cite{CW11}. Of course, by taking the product of our examples with Euclidean space, our work immediately yields  examples also in all dimensions larger than two.

One of the consequences of the examples above is that they reveal a striking difference between the Ricci flows from Theorem \ref{thm:2d-existence} and the classical Ricci flows of Shi.
In particular, 
in the special case that $(\m, g_0)$ is complete and of bounded curvature, then in addition to the Ricci flow $g(t)$ from Theorem \ref{thm:2d-existence}, one also has the Hamilton-Shi Ricci flow $g_1(t)$ with $g_1(0)=g_0$, for $t\in [0,T_1)$, for some $T_1>0$, and by Theorem \ref{thm:2d-existence} (or earlier theory in \cite{GT11}) we must have  $T_1\leq T$ and $g_1(t)=g(t)$ for all $t\in [0,T_1)$, i.e. the classical flow agrees with our flow for as long as the classical flow exists.
We see that Theorems \ref{thm:ex-rf-curv-bursts} and \ref{thm:ex-rf-delayed-unbounded-curv} resolve
the natural question of whether it can actually occur in practice that our flow exists for a longer time interval than Shi's flow. In all the examples above, Shi's flow stops when the curvature blows up, while ours carries on
beyond the time that the curvature becomes unbounded, a property shared with the four-dimensional examples of Cabezas-Rivas and Wilking \cite{CW11} mentioned earlier.
In particular, the traditional use of the term `maximal solution' needs revision.
An additional consequence of our results is then that Shi's flow would not be able to geometrise an arbitrary complete, bounded curvature initial surface that is conformal to a hyperbolic metric $H$, whereas our flow will geometrise even a completely general metric conformal to $H$.

\begin{rmk}
\label{max_stretch}
Recall from \cite{GT11} that the \emph{maximally stretched} condition from part (iii) of Theorem \ref{thm:2d-existence} means
that if $\tilde g(t)$ is any Ricci flow on $\m$ for $t\in [0,\tilde T]$ with $\tilde g(0)\leq g(0)$ (with $\tilde g(t)$ not necessarily complete or of bounded curvature) then $\tilde g(t)\leq g(t)$ for every $t\in[0,\min\{T,\tilde T\}]$. There are several alternative, but ultimately equivalent, ways of writing this condition. For example, one could allow competitors $\tilde g(t)$ only with $\tilde g(0)=g(0)$. Alternatively, one could ask that whenever $0\leq a<b\leq T$ and $\tilde g(t)$ is any Ricci flow on $\m$ for $t\in [a,b)$ with $\tilde g(a)\leq g(a)$ (with $\tilde g(t)$ not necessarily complete or of bounded curvature) then $\tilde g(t)\leq g(t)$ for every $t\in[a,b)$. An inspection of the part of the proof of Theorem \ref{thm:2d-existence} that can be found in \cite{GT11} shows that this apparently stronger property also holds for our flow.
\end{rmk}

\begin{rmk}
Given any Ricci flow $g(t)$, $t\in [0,T)$ arising from Theorem \ref{thm:2d-existence}, and given any $t_0\in [0,T)$, the new Ricci flow $\hat g(t):=g(t+t_0)$ defined for $t\in [0,T-t_0)$ must, by uniqueness (\cite{Top13}, as in Theorem \ref{thm:2d-existence} above), be precisely the unique Ricci flow that Theorem \ref{thm:2d-existence} would produce with initial metric $\hat g(0)=g(t-t_0)$.
Applying this principle to the flow of Theorem \ref{thm:ex-rf-delayed-unbounded-curv}, with $t_0=3$,
i.e. translating the constructed Ricci flow to start at time $t=3$, we see that an initial surface of unbounded curvature need not immediately put itself in the classical situation by becoming of bounded curvature under our flow, but instead can have unbounded curvature for all time. This was originally proved in \cite{GT13} on general surfaces, with a somewhat simpler construction. Specific higher-dimensional examples of Ricci flows with this type of behaviour were constructed by Cabezas-Rivas and Wilking \cite{CW11}.
\end{rmk}

One of the underlying techniques of this paper was introduced in \cite{Top11} where a sequence of complete Ricci flows with locally-controlled curvature was constructed that converged in the Cheeger-Gromov sense to an incomplete Ricci flow. With a great deal of extra technical effort, the same ideas should allow one to construct a Ricci flow with unbounded curvature precisely on an extremely general subset of time $[0,\infty)$. For example, once one has a single Ricci flow that has curvature unbounded precisely on a time interval $[1,a]$ with $a>1$ arbitrarily close to $1$ (and with appropriate spatial asymptotics) then multiple copies of this one flow, with appropriate scaling and translation in time, could be combined together into one connected flow
with unbounded curvature on a given union of closed time intervals.

Finding Ricci flows developing unbounded curvature and maintaining this unbounded curvature as in Theorem \ref{thm:ex-rf-delayed-unbounded-curv} is technically less involved than requiring the curvature to become bounded again as in Theorem \ref{thm:ex-rf-curv-bursts},
and so we will only sketch the proof of Theorem \ref{thm:ex-rf-delayed-unbounded-curv}.

\emph{Acknowledgements:} This work was partially supported by The
Leverhulme Trust, EPSRC Programme grant EP/K00865X/1 and the SFB 878: Groups, Geometry \& Actions.

\section{Strategy of the proof}
\label{strategy_sect}

We will prove Theorem \ref{thm:ex-rf-curv-bursts} by making a very precise construction of a suitable Ricci flow, with precise constants, which is somewhat technical in parts. However it is important to digest the picture behind our construction first, and this may be enough to understand fully what is going on, without further technicality.

\begin{figure}[h]
\centering
\includegraphics{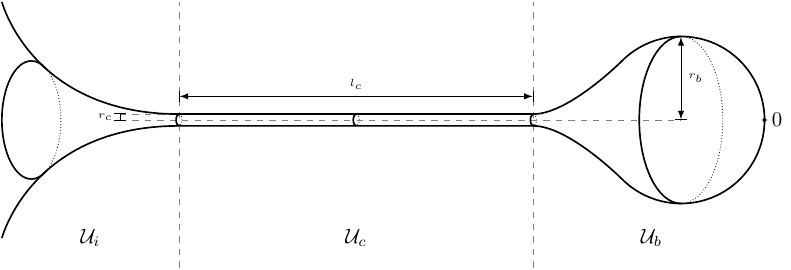}
\caption{\textsc{cb}-surface -- `lollipop': Cylinder with bulb cap}
\label{fig:g-cb}
\end{figure}

The basic building block of our construction is a `lollipop' surface, pictured in Figure \ref{fig:g-cb}, that consists of a plane that has had a local region drawn out into a long thin cylinder ($\mathcal{U}_c$ in the figure) with a bulb ($\mathcal{U}_b$) on the end. This can easily be constructed (Section \ref{ssect:construction}) to have curvature bounded uniformly above and below whilst allowing the cylinder to be as thin as we like. The length of the cylinder is chosen so that its area is of order one (i.e. the cylinder is long and thin) and the bulb on the end will be of a uniform size and area.

Consider now what happens to such a surface under Ricci flow when we flow it (for all time) using Theorem \ref{thm:2d-existence}. Because the curvature is initially bounded, and our flow initially agrees with the flow of Shi, the curvature of the flow will remain bounded for a uniform time (independent of how thin the cylinder is). The cylinder, being flat, will essentially try to remain a cylinder. (This can be made precise quickly by using Perelman's Pseudolocality Theorem \cite{Per02}, but here we will use barrier arguments.) Meanwhile, the bulb part of the surface will shrink, something like a shrinking round sphere, losing area at a constant rate. After a time proportional to the initial area of the bulb (which is controlled uniformly, independently of how small the radius of the cylinder might be) the Ricci flow will now look similar to how it looked initially, up to and including the long thin cylinder, but now instead of having a big bulb attached to the end, we will show that it will now have a small cap attached ($\to$ Figure \ref{fig:g-cb2}).

\begin{figure}[h]
\centering
\includegraphics{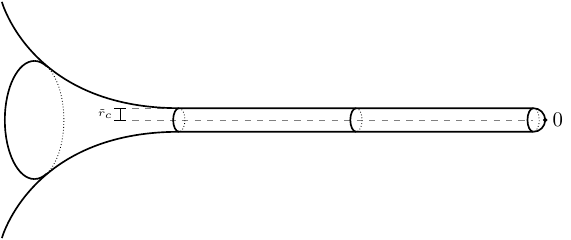}
\caption{\textsc{cb}-surface at some later time: Cylinder with (small) cap }
\label{fig:g-cb2}
\end{figure}

This point in time marks a transition for the flow. From now on, for a uniformly controlled time, the capped cylinder will evolve somewhat like a cigar soliton (\cite[\S 1.2.2]{Top06}). It will keep on looking like a capped cylinder, but the length of that cylinder will shrink at a constant rate. During this phase, if the `radius' of the cylinder is $r_c$, the curvature of the flow should be of the order of $r_c^{-2}$, i.e. large.

In a uniformly controlled amount of time, the cylinder will disappear entirely, and the flow will look like a plane with a truncated hyperbolic cusp attached ($\to$ Figure \ref{fig:g-cb3}) in some local region.
This moment marks a second phase transition for the flow. Following \cite{Top12} we show that this hyperbolic cusp contracts in a uniformly controlled way (independently of how thin the cylinder was, and thus how long the cusp is) and the curvature returns to being controlled independently of the radius $r_c$.

\begin{figure}[h]
\centering
\includegraphics{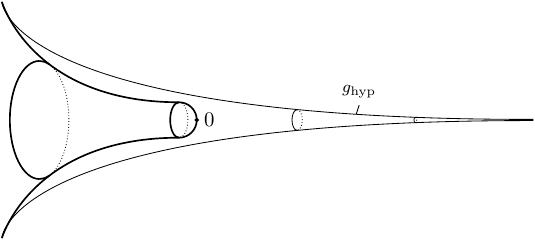}
\caption{\textsc{cb}-surface at some further later time: Cusp with cap}
\label{fig:g-cb3}
\end{figure}

What this construction yields is almost what is required in Theorem \ref{thm:ex-rf-curv-bursts}, except instead of the curvature being infinite on the intermediate time interval, it is of order $r_c^{-2}$. However, by gluing together infinitely many copies of this basic construction, with smaller and smaller radii $r_c$, our job is done. Note that it turns out not to be too difficult, using pseudolocality technology, to show that the different copies of the basic building block described above do not influence each other too much.

Of course there are a number of things to check to be sure that this construction can be carried through. Barrier arguments turn out to be very useful, and some of these can be recycled from \cite{GT13, Top12, Top11}. However, barrier arguments alone did not seem to be sufficient to force the bulb part of the construction above to shrink to nothing in a controlled time. Instead, in Section \ref{width_sect}  we introduce a novel and very simple technique where we combine the Ricci flow with the double-speed mean curvature flow (i.e. the curve shortening flow) to derive the required width estimate. Roughly, we allow a small loop in the cylinder part $\mathcal{U}_c$ to flow over the bulb part $\mathcal{U}_b$, arguing that it must achieve this in a controlled amount of time without its length increasing significantly. The evolving curve acts as a noose, squashing thin the bulb.

Finally, we remark that the Ricci flow of the surface we describe here is also a useful example in the study of Harnack inequalities for Ricci flow, as we will describe elsewhere.

\section{Construction}
\label{ssect:construction}

In order to define a metric like in Figure \ref{fig:g-cb} let us first describe the required ingredients and then put them together explicitly in Definition \ref{defn:cb-metric}. Starting with a flat cylinder ($\to\,\mathcal{U}_c$) of radius $r_c$ and length $l_c$, we cap it off on one side by a bulb ($\to\,\mathcal{U}_b$) of radius $r_b=\sqrt2$, and on the other side we interpolate ($\to\,\mathcal{U}_i$) it with the flat (complex) plane $\bigl(\mathbb C,|\dz|^2\bigr)$. In order to do so, we are going to connect the cylinder with hyperbolic approximated cusps on both sides.  The one in $\mathcal{U}_b$ can be joined to a round sphere and the other one in $\mathcal{U}_i$ to the flat plane $\mathcal{U}_e=\mathbb C\setminus\overline{\mathcal{U}_i\cup\mathcal{U}_c\cup\mathcal{U}_b}$. To see that each transition is at least differentiable we are going to specify piecewisely the conformal factor $\ee^{2u(s,\theta)}\bigl(\ds^2+\df\theta^2\bigr)$ of each building block in logarithmic cylindrical coordinates $(s,\theta)\in\mathbb R\times[0,2\pi)$:
\begin{align}
\tag{\text{flat plane}}
\mathbb R\ni s &\mapsto -s \\
\tag{\text{cylinder of radius $r_c>0$}}
\mathbb R\ni s &\mapsto \log r_c \\
\tag{\text{round sphere of radius $r_b=\sqrt{2}$}}
\mathbb R\ni s &\mapsto -\log\cosh(s) + \log r_b \\
\tag{\text{hyperbolic approximated cusp}}   \textstyle\left(-\frac\pi{2r_c},\frac\pi{2r_c}\right)\ni s &\mapsto   -\log\left(r_c^{-1}\cos(r_cs)\right)
\end{align}

\begin{figure}[h]
\centering
\includegraphics{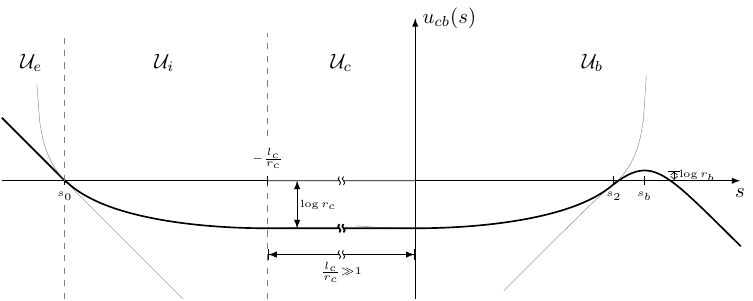}
\caption{Conformal factor of a cylinder with bulb cap in cylindrical     coordinates}
\label{fig:u-cb}
\end{figure}

\begin{defn}
\label{defn:cb-metric}
For a given length $l_c>0$ and radius $r_c\in(0,1)$ the  \textsc{cb}-surface is the rotationally symmetric plane  $\bigl(\mathbb C,g_{\mathsf{cb}}\bigr)$ where  $g_{\mathsf{cb}}=\ee^{2u_{\mathsf{cb}}(s)}\bigl(\ds^2+\df\theta^2\bigr)$ has in logarithmic cylindrical coordinates $z=\ee^{-s+s_e+\mathrm{i}\theta}$ the form
\begin{equation}
\label{eq:def-u_cb}
\setlength{\arraycolsep}{.7ex}
u_{\mathsf{cb}}(s) = \left\{\begin{array}{llll}
u_1(s) =-s + s_e& \text{for $s\in(-\infty,s_0]$}
&&\simeq\overline{\mathcal{U}_e} \\
u_2(s) = -\log\left(r_c^{-1} \cos(r_cs+l_c) \right)
& \text{for $s\in\left(s_0,-\nicefrac{l_c}{r_c}\right)$}
&& \simeq\mathcal{U}_i\\
u_3(s) = -\log r_c^{-1} & \text{for         $s\in\left[-\nicefrac{l_c}{r_c},0\right]$}
&& \simeq\overline{\mathcal{U}_c}\\
u_4(s) = -\log\left(r_c^{-1} \cos(r_cs) \right) & \text{for
$s\in\bigl(0,s_2\bigr]$}
&\multirow{2}{*}{$\biggr\}$}
&\multirow{2}{*}{$\simeq\mathcal{U}_b\setminus\{0\}$}\\
u_5(s) = -\log\cosh(s-s_b) + \frac12\log2 & \text{for                             $s\in(s_2,\infty)$},
\end{array}\right.
\end{equation}
with parameters   $s_0\in\left(-\frac{l_c+\frac\pi2}{r_c},-\frac{l_c}{r_c}\right)$,   $s_e\in\left(s_0,-\frac{l_c}{r_c}\right)$,   $s_2\in\left(0,\frac\pi{2r_c}\right)$ and $s_b>s_2$ chosen uniquely such that $u_{\mathsf{cb}}\in\Cts^1(\mathbb R)$.

For these parameters $r_c$ and $l_c$, we define the \textsc{cb}-Ricci flow to be the complete Ricci flow starting from the \textsc{cb}-surface, that is given by Theorem \ref{thm:2d-existence}.
\end{defn}

Taking a geometric viewpoint, it is easy to verify the existence of the \textsc{cb}-metric without computation -- essentially one takes the three middle parts ($u_2$, $u_3$, $u_4$) stretching from $-\frac{l_c+\nicefrac\pi2}{r_c}$ all the way to $\frac\pi{2r_c}$ (which automatically combine to give a $\Cts^1$ function) and then moves in the graph of $s\mapsto -s$ from the left until it touches, and moves the graph $s\mapsto -\log\cosh(s) + \frac12\log2$ in from the right until it touches.

\begin{lemma}[\textsf{Properties of \textsc{cb}-surface}]
\label{lemma:cb-prop}
For given length $l_c>0$ and radius $r_c\in\left(0,\nicefrac1{10}\right)$ the \textsc{cb}-surface   $\bigl(\mathbb C,g_{\mathsf{cb}}\bigr)$ has the following properties:
\begin{compactenum}[(i)]
\item $\GK_{g_{\mathsf{cb}}}\in\bigl[-1,\nicefrac12\bigr]$ a.e.
\item $s_0=-r_c^{-1}\left(l_c+\arctan r_c^{-1}\right)$ and       $s_e=s_0+\frac12\log\bigl(1+ r_c^2\bigr)$; in particular
\[
\Disc_{\sqrt{1+r_c^2}}=\mathcal{U}_i\cup\overline{\mathcal{U}_c}\cup\mathcal{U}_b \]
\item $s_b\le\frac74r_c^{-1}$
\item $\Vol_{g_{\mathsf{cb}}}\mathcal{U}_b < 10\pi$
\end{compactenum}
\end{lemma}

\begin{figure}[t]
\centering
\includegraphics{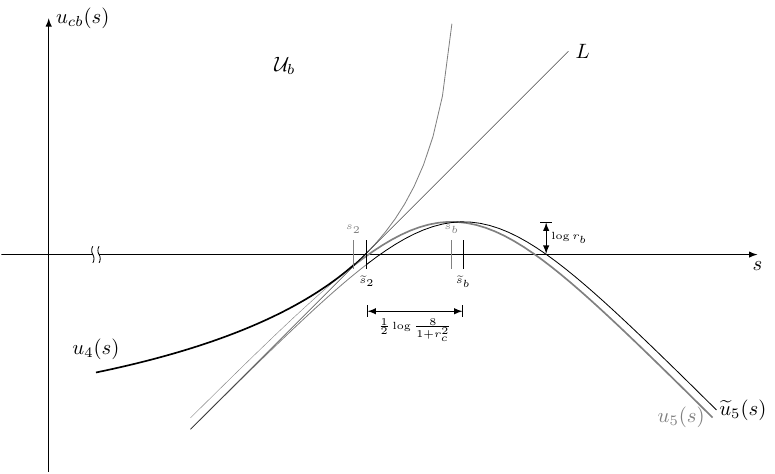}
\caption{Estimate of $s_b$ in Lemma \ref{lemma:cb-prop}}
\label{fig:estitmate-ub}
\end{figure}

\begin{proof}
(i) follows from the fact that the curvatures of the segements in   \eqref{eq:def-u_cb} take the constant values $0$, $-1$, $0$, $-1$   and $\frac12$ respectively. Since $u_{\mathsf{cb}}\in\Cts^1(\mathbb   R)$, we compute $s_0=-r_c^{-1}\left(l_c+\arctan r_c^{-1}\right)$ by   solving the equation $u_2'(s_0)=u_1'\equiv-1$ for   $s_0\in\left(-\frac{l_c+\frac\pi2}{r_c},-\frac{l_c}{r_c}\right)$.   Consequently, we also have $s_e=s_0 + u_1(s_0) = s_0 + u_2(s_0) =s_0   + \frac12\log\left(1+r_c^2\right)$, i.e. (ii).

Instead of an explicit formula for $s_b$ (and $s_2$), it suffices to get upper bounds, which we estimate as follows: As shown in Figure \ref{fig:estitmate-ub} we shift $u_5$ slightly to the right such that a straight line $L$ parallel to $s\mapsto s$ fits exactly inbetween $u_4$ and the shifted $u_5$. Let us define $\widetilde s_b$ by denoting the shifted $u_5$ as 
$$\widetilde u_5(s) = -\log\cosh(s-\widetilde s_b) + \frac12\log2.$$
The line $L$ touches $u_4$ at $\widetilde s_2:=r_c^{-1}\arctan r_c^{-1}$ (i.e. $\widetilde s_2$ is a solution of $u_4'(\widetilde s_2)=1$), hence the explicit formula for $L$ is $s\mapsto s-\widetilde s_2+u_4(\widetilde s_2)= s-\widetilde s_2 +\frac12\log\left(1+r_c^2\right)$.
This determines $\widetilde s_b:=\widetilde s_2 -\frac12\log\left(1+r_c^2\right) +\frac32\log2 = \widetilde s_2 +\frac12\log\frac8{1+r_c^2}$. Note that the graph of $u_4$ lying above $L$ means that
$$u_4(s) \ge\textstyle s-\widetilde s_2+\frac12\log\left(1+r_c^2\right) \qquad\textstyle\text{for all }s\in\left[0,\frac\pi{2r_c}\right)$$
while $L$ lying above $\widetilde u_5$ means that 
$$\textstyle s-\widetilde s_2 +\frac12\log\left(1+r_c^2\right)
\ge -\textstyle\log\cosh(s-\widetilde s_b)+\frac12\log2 \qquad \text{for all }s\in[0,\infty).$$
Because $s_b<\widetilde s_b$ (we shifted $u_5$ to the right by $\widetilde s_b - s_b$) we can estimate
\begin{align}
\label{eq:sb-tsb-estimate}
s_b < \widetilde s_b &= \widetilde s_2 +\frac12\log\frac8{1+r_c^2}   \le \widetilde s_2 + \frac32\log2 \le \frac\pi{2r_c} + \frac32\log2   \le \frac74r_c^{-1},
\end{align}
because $r_c<\frac{1}{10}$, thus proving (iii).  Since $u_5'(s)<1$ for all $s\in\mathbb R$ and $u_4'(s)\ge1$ for all $s\ge\widetilde s_2$, we know that $s_2<\widetilde s_2$, and also $u_5(s)\le u_4(s)$ for all $s\in[s_2,\widetilde s_2]$. Therefore,
\begin{align*}
\Vol_{g_{\mathsf{cb}}}\mathcal{U}_b
&= 2\pi\int_0^{s_2} \ee^{2u_4(s)}\,\ds
+ 2\pi\int_{s_2}^\infty \ee^{2u_5(s)}\,\ds \\
&\le 2\pi\int_0^{s_2} \ee^{2u_4(s)}\,\ds +   2\pi\int_{s_2}^{\widetilde s_2} \ee^{2u_4(s)}\,\ds
+ 2\pi\int_{\widetilde s_2}^\infty \ee^{2u_5(s)}\,\ds \\
&= 2\pi\int_0^{\widetilde s_2}   \frac{r_c^2}{\bigl(\cos(r_cs)\bigr)^2}\, \ds + 2\pi\int_{\widetilde     s_2}^\infty
\frac2{\bigl(\cosh(s-s_b)\bigr)^2}\,\ds \\
&= 2\pi + 4\pi\bigl(1 + \tanh(s_b-\widetilde s_2)\bigr)   \stackrel{\eqref{eq:sb-tsb-estimate}}{<} 2\pi \left(3 +     2\tanh\left(\frac32\log2\right)\right) = \frac{82}9\pi < 10\pi.
\end{align*}
concluding (iv).
\end{proof}

Note that the essential point in (iv) is that we have a uniform upper bound for the volume that is independent of $r_c$, and this fact is essentially obvious from the construction if one thinks geometrically.

\section{Width estimate using curve shortening flow}
\label{width_sect}

Under Ricci flow, the bulb part of the \textsc{cb}-surface will shrink, and intuitively the area will shrink at a constant rate of about $4\pi$, independently of $r_c$. The precise result we need is that after a definite amount of time, the `width' of the bulb is similar to the width of the cylinder, and to achieve this we use a combination of the curve shortening flow and the Ricci flow.

\begin{prop}
\label{prop:finite-csf}
On a surface $\Mf^2$ let $\mathcal{U}_0\Subset\Mf$ be a simply
connected domain with smooth boundary $\partial\mathcal{U}_0$.
For $T=\frac1{4\pi}\Vol_{g(0)}\mathcal U_0$
let $\bigl(g(t)\bigr)_{t\in[0,T]}$ be a smooth Ricci flow on
$\Mf$ and $\bigl(\gamma(t)\bigr)_{t\in[0,T)}$ be a smooth embedded solution to the
(double speed) curve shortening flow within $\bigl(\Mf,g(t)\bigr)$
starting from the closed curve along $\partial\mathcal{U}_0$, i.e.
\begin{equation}
\label{eq:double-csf}
\left\{
\begin{aligned}
\pddt\gamma(t) &= -2\kappa_{g(t),\gamma(t)}\nu_{g(t),\gamma(t)}\\
\im \gamma(0) &= \partial\mathcal U_0,
\end{aligned}\right.
\end{equation}
where $\nu$ is the `outward' unit normal to $\gamma$ and $\kappa$ is the corresponding geodesic curvature.
Then $\Vol_{g(t)}\mathcal{U}_t =4\pi(T-t)$ for all
$t\in[0,T]$, where $\mathcal{U}_t$ is the simply connected
domain with boundary $\partial\mathcal{U}_t=\im\gamma(t)$ (on the same side as $\mathcal{U}_0$).
\end{prop}
\begin{proof}
We calculate using the generalised Leibniz integral rule, the
Gauss-Bonnet Theorem and the fact that 
$\frac{d}{dt}\dmu_{g(t)} = -2 \GK_{g(t)}\dmu_{g(t)}$ 
(see \cite[(2.5.7)]{Top06})
\begin{align*}
\ddt \int_{\mathcal{U}_t} \dmu_{g(t)} &= -2\int_{\mathcal{U}_t}
\GK_{g(t)}\dmu_{g(t)} + \int_{\im\gamma(t)}
\left\langle\pddt\gamma(t),\nu_{g(t),\gamma(t)}\right\rangle_{\!\!g(t)}
\,\ds\\
&= -4\pi\mathop\chi(\mathcal U_t) + 2\int_{\partial\mathcal U_t}
\kappa_{g(t)}\,\ds - \int_{\im\gamma(t)} 2\kappa_{g(t),\gamma(t)}
\,\ds = -4\pi.
\end{align*}
Therefore, $\Vol_{g(t)}\mathcal U_t = \Vol_{g(0)}\mathcal U_0
- 4\pi t = 4\pi(T-t)$ for all $t\in[0,T]$.
\end{proof}

This simple principle will be behind the following width estimate.
It will control the length of any circle at \emph{some} time, and then our lower curvature bounds will give control on its length for all later times, and in particular at the time $T$ that we expect the bulb to be extinguished.

\begin{lemma}
\label{lemma:width-estimate}
For some $r_c>0$ and $l_c>1$ let
$\bigl(g(t)\bigr)_{t\in[0,\infty)}$ be the \textsc{cb}-Ricci flow (i.e. the complete Ricci flow on $\mathbb C$ from Theorem
\ref{thm:2d-existence} starting from the corresponding \textsc{cb}-metric). Then we have the `width'
estimate
\begin{equation}
\label{eq:width-estimate}
\width_{g(T)}
\left(\mathcal{U}_c\cup\overline{\mathcal{U}_b}\right)
:= \max_{r\in[0,l_b+l_c]} \L_{g(T)}\partial\gBall_{g(0)} (0;r)
\le 2\pi\sqrt{2T+1}\,r_c,
\end{equation}
where $T=\frac1{4\pi}\Vol_{g(0)}\mathcal{U}_b$,
$l_b=\dist_{g(0)}\bigl(0,\partial\mathcal{U}_b\bigr)$ and
$\L_{g(T)}$ represents the length.
\end{lemma}
In cylindrical coordinates using the convention from
\S\ref{ssect:construction} and writing
$g(t)=\ee^{2u(t,s)}\bigl(\ds^2+\df\theta^2\bigr)$, we can state the
`width' estimate
\eqref{eq:width-estimate} in terms of the conformal factor:
\begin{equation}
\label{eq:width-estimate-u}
\tag{\ref{eq:width-estimate}'}
u(T,s) \le \log r_c + \frac12\log\bigl(2T+1\bigr)
\qquad\text{for all $\textstyle s\in\left(-\frac{l_c}{r_c},\infty\right)$.}
\end{equation}

\begin{proof}
Let $\bigl(\gamma(t)\bigr)_{t\in[0,T)}$ be the rotationally symmetric  solution to \eqref{eq:double-csf} starting from
$\im\gamma(0)=\partial\mathcal{U}_b$.
Define a
time-dependent radius $\rho:[0,T)\to(0,\infty)$ such that
$\partial\gBall_{g(0)}\bigl(0;\rho(t)\bigr)=\im\gamma(t)$ for all
$t\in[0,T)$. Now estimate using Chen's Theorem
\ref{thm:chen-lower-curv-bd} and the fact that $\GK_{g(0)}\geq -1$
\begin{align*}
\ddt \L_{g(t)}\gamma(t) &= - \int_{\im\gamma(t)}
\left( \GK_{g(t)} +2\kappa_{g(t)}^2\right) \,\ds
\le \frac1{2t+1} \L_{g(t)}\gamma(t),
\end{align*}
which we integrate to
$\L_{g(t_2)}\gamma(t_2) \le \frac{\sqrt{2t_2+1}}{\sqrt{2t_1+1}}\,
\L_{g(t_1)}\gamma(t_1)$ for all $0\le t_1<t_2<T$.
Similarly, for each fixed $t_0\in [0,T)$,
\begin{align}
\label{banana}
\ddt \L_{g(t)}\gamma(t_0) &= - \int_{\im\gamma(t_0)}
\GK_{g(t)} \,\ds
\le \frac1{2t+1} \L_{g(t)}\gamma(t_0),
\end{align}
which integrates
$\L_{g(t_2)}\gamma(t_0) \le \frac{\sqrt{2t_2+1}}{\sqrt{2t_1+1}}\,
\L_{g(t_1)}\gamma(t_0)$ for all $0\le t_1<t_2<T$.

Fix any radius $r\in(0,l_b)$ at which we would like to estimate the width. By Proposition \ref{prop:finite-csf} (and
symmetry) we see that
$\rho(t)\searrow0$ as $t\nearrow T$, so by continuity we find
a time $t_r:=\max\bigl\{t\in[0,T]:\im\gamma(t)
= \partial\gBall_{g(0)}(0;r)\bigr\}$ when the curve coincides with
the boundary of the $r$-ball. Hence, we estimate
\begin{align*}
\L_{g(T)}\partial\gBall_{g(0)} (0;r) &\le
\sqrt{\frac{2T+1}{2t_r+1}}\,\L_{g(t_r)}
\partial\gBall_{g(0)} (0;r)
=\sqrt{\frac{2T+1}{2t_r+1}}\,\L_{g(t_r)}
\gamma(t_r)\\
&\le \sqrt{2T+1}\,\L_{g(0)}\gamma(0)
= \sqrt{2T+1}\,2\pi\, r_c.
\end{align*}
Finally, the principle behind \eqref{banana} shows that $\L_{g(T)}\partial\gBall_{g(0)} (0;r) \le
2\pi\sqrt{2T+1}\,r_c$ for all $r\in[l_b,l_b+l_c]$.
\end{proof}

\section{Barriers}
\label{barrier_sect}

First we show that the bulb does not shrink too fast.

\begin{lemma}[\textsf{Coarse lower sphere barrier}]
\label{lemma:coarse-lower-sphere-barrier}
For any parameters $l_c>1$ and $r_c\in(0,1)$ let
$\bigl(\ee^{2u(t)}(\ds^2+\df\theta^2)\bigr)_{t\in[0,\infty)}$ be
the corresponding \textsc{cb}-Ricci flow.
Then we have the lower barrier
\begin{equation}
\label{eq:coarse-lower-sphere-barrier}
u(t,s) \ge -\log\cosh(s-s_b) + \frac12\log2(1-t)
\end{equation}
for all $(t,s)\in\left[0,1\right)\times\mathbb R$.
\end{lemma}
\begin{proof}
Comparing the solution to an `inner' sphere
in the bulb part, i.e. $u_5(s)$ in Definition \ref{defn:cb-metric}
for the initial condition, the result follows from the fact that
the \textsc{cb}-Ricci flow
is  maximally stretched (see Theorem \ref{thm:2d-existence}).
\end{proof}

\subsection{Cigar barriers}

Hamilton's cigar will serve as a useful barrier. Therefore,
like in \cite[\S2.2.2]{Top11} we introduce the notation
of the standard cigar, in logarithmic cylindrical coordinates
\begin{equation}
\label{eq:cigar}
\mathcal{C}(s) := -\frac12\log\left(\ee^{2s}+1\right)
\end{equation}
with the corresponding Ricci (soliton) flow
\begin{equation}
\label{eq:cigar-flow}
(t,s) \longmapsto \mathcal{C}\bigl(s+2t\bigr).
\end{equation}
In practice, we will need rescaled (and translated) forms of the cigar
(and its associated Ricci flow),
so for $\lambda>0$, we define
\begin{equation}
\label{eq:scaled-cigar}
\mathcal{C}_\lambda(s) := \mathcal{C}(s) -\frac12\log\lambda
\qquad\text{and}\qquad
\mathcal{C}_\lambda(t,s) = \mathcal{C}_\lambda\bigl(2\lambda t+s\bigr).
\end{equation}
We have the rough estimates
\begin{equation}
\label{eq:cigar-estimate}
\left.
\begin{array}{ll}
-\frac12\log2 & \text{for $s\le0$}\\
-\frac12\log2 -s & \text{for $s>0$}
\end{array}\right\} \le \mathcal{C}(s) \le \left\{
\begin{array}{ll}
0 & \text{for $s\le0$}\\
-s & \text{for $s>0$}
\end{array}\right.
\end{equation}
and a (unit) sphere is dominated by the cigar:
\begin{equation}
\label{eq:cigar-bulb-estimate}
-\log\cosh(s) \le
\textstyle\mathcal{C}_{\frac14}(s)
\qquad\text{for all $s\in\mathbb R$.}
\end{equation}
Thinking of the cigar solution as a capped cylinder which `translates' in time, we
expect its area to behave like the area of a cylinder of length $-s-2t$,
i.e.~$\sim2\pi(-s-2t)$ for $s<-2t$; more precisely, we have the lower
estimate for all $s<-2\lambda t$
\begin{align}
\label{eq:cigar-area-t}
\Vol_{\mathcal{C}_\lambda(t,\cdot)} \Bigl([s,\infty)\times[0,2\pi) \Bigr)
&= 2\pi\lambda^{-1}\,
\frac12 \log\left(\ee^{-2(2\lambda t+s)}+1\right)
\ge -2\pi\lambda^{-1}\bigl(2\lambda t+s\bigr).
\end{align}
Note, that we are abusing notation by writing geometrical quantities
with respect to the conformal factor instead of the associated metric.

\begin{lemma}[\textsf{Coarse upper cigar barrier}]
\label{lemma:coarse-upper-cigar-barrier}
For some parameters $l_c>1$ and $r_c\in(0,1)$ let
$\bigl(\ee^{2u(t)}(\ds^2+\df\theta^2)\bigr)_{t\in[0,\infty)}$ be the
corresponding \textsc{cb}-Ricci flow.
Then we have for all $(t,s)\in\bigl[0,r_c^{-1}\bigr]\times
\left[-\frac{l_c}{r_c},\infty\right)$
\begin{equation}
\label{eq:coarse-upper-cigar-barrier}
u(t,s) \le \mathcal{C}_{\frac18}(t,s-s_b).
\end{equation}
\end{lemma}
\begin{proof}
Choosing cylindrical coordinates as usual, we can compare the
initial metric to a cigar dominating the bulb part
\begin{align*}
u(0,s) &\le \left\{
\begin{array}{ll}
\frac12\log 2 & \text{if } s\in\left[-\frac{l_c}{r_c},s_b\right] \\
\frac12\log 2 -\log\cosh(s-s_b)\hspace{3.5em} & \text{if } s\in(s_b,\infty)
\end{array} \right.\\
&\le \frac12\log 2 +\mathcal{C}_{\frac14}(s-s_b) =\mathcal{C}_{\frac18}(s-s_b)
\qquad\text{for all $\textstyle
s\in\left[-\frac{l_c}{r_c},\infty\right)$}
\end{align*}
using \eqref{eq:cigar-estimate} and \eqref{eq:cigar-bulb-estimate} in
the second line. 
On the other hand, using Corollary \ref{cor:chen-lower-curv-bd-cf}
we have for all $t\in\bigl[0,r_c^{-1}\bigr]$ at the boundary
$s=-\frac{l_c}{r_c}$:
\begin{align*}
\textstyle u\left(t,-\frac{l_c}{r_c}\right) &\le
{\textstyle u\left(0,-\frac{l_c}{r_c}\right)} +\frac12\log(2t+1)
= \frac12\log\left(2r_c^2t+r_c^2\right) \le \frac12\log4\\
&\le \frac12\log8 + \textstyle \mathcal{C}\left(\frac{t}{4}-\frac{l_c}{r_c}-s_b\right)
= \mathcal{C}_{\frac18}\left(t,-\frac{l_c}{r_c}-s_b\right)
\quad\text{using \eqref{eq:cigar-estimate}.}
\end{align*}
Hence, we may apply the comparison principle and obtain the result.
\end{proof}

That upper cigar barrier is coarse in the sense that the `circumference' of the cigar is of order one so that it can dominate the entire bulb. We now make a refinement than says the bulb part is dominated by a cigar of tiny circumference of order $r_c$ once the bulb has been given time to collapse. We also get lower bounds which show that the cylinder itself has not collapsed at that time.

\begin{lemma}[\textsf{Refined cigar barriers}]
\label{lemma:refined-cigar-barriers}
For parameters $r_c\in\left(0,\frac1{10}\right)$ and
$l_c\ge \frac18 r_c^{-1}$ there exists a time
$t_1\in\bigl(\nicefrac34,\nicefrac52\bigr)$, such that
the corresponding \textsc{cb}-Ricci flow
$\bigl(\ee^{2u(t)}(\ds^2+\df\theta^2)\bigr)_{t\in[0,\infty)}$
has the upper and lower barriers
\begin{align}
\label{eq:upper-cigar-barrier}
u(t,s) &\le \mathcal{C}_{(4r_c)^{-2}}\bigl(t-t_1,s-2r_c^{-1}\bigr)
&&\text{for all
$\textstyle
(t,s)\in\bigl[t_1,t_1+1\bigr]\times
\left[-\frac{l_c}{r_c},\infty\right)$,}\\
\label{eq:lower-cigar-barrier}
u(t,s) &\ge \textstyle\mathcal{C}_{r_c^{-2}}\left(t-t_1,s\right)
&&\text{for all
$\textstyle
(t,s)\in[t_1,\infty)\times\mathbb R$.}
\end{align}
\end{lemma}
\begin{proof}
Fix $r_c$ and $l_c$ according to the lemma's
statement.
First we find the time that the bulb has first collapsed to a specified degree.
\begin{description}
\item[Claim.] There exist a time
$t_1\in\left(\frac34,\frac52\right)$
and a
point $s_1>0$ such that
\begin{align}
\label{eq:refined-cigars-claim-1}
u(t_1,s_1)&=\max_{s\ge-\frac{l_c}{r_c}}
u(t_1,s) = \log r_c + \frac12\log8 \\
\label{eq:refined-cigars-claim-2}
\text{and}\qquad
u(t_1,s) &\ge \begin{cases}
\log r_c & \text{for all $s\in(-\infty,s_1]$}\\
s_1-s +\log r_c & \text{for all $s\in(s_1,\infty)$.}
\end{cases}
\end{align}
\item[\textit{\mdseries Proof of Claim.}]
Note first that given any $t_1\leq \frac52$, we cannot have $s_1\in\bigl[-\frac{l_c}{r_c},0\bigr]$ satisfying \eqref{eq:refined-cigars-claim-1}
because Corollary \ref{cor:chen-lower-curv-bd-cf} would constrain
$u(t_1,s_1)\leq u(0,s_1)+\frac12\log(2t_1+1)\leq \log r_c+\frac12\log 6$
(i.e. the cylinder cannot fatten up too quickly).
By Lemma
\ref{lemma:cb-prop}(iv), we know
$\frac1{4\pi}\Vol_{g_{\mathsf{cb}}}\mathcal{U}_b<\frac52$, thus the
existence of
$t_1\in\left(0,\frac1{4\pi}\Vol_{g_{\textsf{cb}}}\mathcal{U}_b\right)
\subset\left(0,\nicefrac52\right)$
and $s_1>0$ satisfying \eqref{eq:refined-cigars-claim-1} is a
consequence of Lemma \ref{lemma:width-estimate} and in particular estimate
\eqref{eq:width-estimate-u}. From Lemma
\ref{lemma:coarse-lower-sphere-barrier} we see that
$t_1\ge1-4r_c^2\ge\frac34$. In order to
show \eqref{eq:refined-cigars-claim-2} we are going to compare
$u\bigr|_{[0,t_1]}$ with the flat, static cylinder solution
$(t,s)\mapsto\log r_c$ on $[0,t_1]\times(-\infty,s_1]$ and with
the flat plane $(t,s)\mapsto s_1-s+\log r_c$ on
$[0,t_1]\times[s_1,\infty)$:

First observe, that $u(t,s_1)>\log r_c$ for all $t\in[0,t_1]$
because if at some time $t_0\in[0,t_1)$ we had $u(t_0,s_1)\le\log
r_c$, then using Corollary \ref{cor:chen-lower-curv-bd-cf} along with
\eqref{eq:refined-cigars-claim-1} we would obtain the contradiction
\begin{align*}
\log r_c+\frac12\log8 &= u(t_1,s_1)\\
&\le u(t_0,s_1) +
\frac12\log\frac{2t_1+1}{2t_0+1}\\
&\le \log r_c +
\frac12\log6. 
\end{align*}
By construction we have initially
\[ u(0,s)=u_{\mathsf{cb}}(s)\ge \begin{cases}
\log r_c & \text{for all $s\in(-\infty,s_1]$}\\
s_1-s +\log r_c & \text{for all $s\in(s_1,\infty)$.}
\end{cases} \]
Therefore, we can apply a comparison principle on both sides of
$s_1$ and conclude   \eqref{eq:refined-cigars-claim-2}. \hfill //
\end{description}
In order to show \eqref{eq:upper-cigar-barrier}, we start
comparing both solutions `initially' at time
$t_1$. Combining the coarse upper barrier from Lemma
\ref{lemma:coarse-upper-cigar-barrier}
(valid until time $r_c^{-1}\ge 10>\nicefrac52>t_1$)
with
\eqref{eq:refined-cigars-claim-1} and using estimate
\eqref{eq:cigar-estimate} we have for all $s\ge-\frac{l_c}{r_c}$
\begin{align*}
u(t_1,s)
&\le \min\left\{\log r_c+\frac12\log8,\,
\mathcal{C}_{\frac18}(t_1,s-s_b)\right\} \\
&\le \left\{
\begin{array}{cl}
\frac12\log\bigl(8r_c^2\bigr)
& \text{for $s\in\left[-\frac{l_c}{r_c},
\widetilde{s_1}\right)$} \\
-s + s_b - \frac{t_1}4 +\frac12\log8
\hspace{5em}& \text{for $s\in[\widetilde{s_1},\infty)$}
\end{array} \right.\\
&= -\frac12\log2 + \left\{
\begin{array}{ll}
\frac12\log\bigl(16r_c^2\bigr)
& \text{for $s\in\left[-\frac{l_c}{r_c},
\widetilde{s_1}\right)$} \\
\frac12\log(16r_c^2) -s + \widetilde s_1
\hspace{1em}& \text{for $s\in[\widetilde{s_1},\infty)$}
\end{array} \right.\\
&\le \mathcal{C}_{\left(4r_c\right)^{-2}}\bigl(s -
2r_c^{-1}\bigr)
\hspace{7em}\textstyle\text{ for all
$s\in\left[-\frac{l_c}{r_c},\infty\right)$,}
\end{align*}
where $\widetilde{s_1}=s_b -\frac{t_1}4 -\log r_c\le
2r_c^{-1}$ (using Lemma \ref{lemma:cb-prop}(iii)) is roughly
the intersection of $\frac12\log\bigl(8r_c^2\bigr)$ with
$\mathcal{C}_{\frac18}(t_1,s-s_b)$. At the boundary
$s=-\frac{l_c}{r_c}$ we have for all $t\in[0,t_1+1]$ (using
Corollary \ref{cor:chen-lower-curv-bd-cf} and \eqref{eq:cigar-estimate})
\begin{align*}
u\left(t,-\frac{l_c}{r_c}\right) &\le \log r_c +
\frac12\log(2t+1) \le \frac12\log\left(8r_c^2\right) \\
&\le \frac12\log(16r_c^2) +\textstyle
\mathcal{C}\left(2(16r_c^2)^{-1} -\frac{l_c}{r_c} -2r_c^{-1}\right)\\
&= \textstyle\mathcal{C}_{(4r_c)^{-2}}\left(
1,-\frac{l_c}{r_c}-2r_c^{-1}\right)
\le \mathcal{C}_{(4r_c)^{-2}}
\left(t-t_1,-\frac{l_c}{r_c}-2r_c^{-1}\right)
\end{align*}
with the above choice of $l_c\ge\frac18 r_c^{-1}$.
Thus the comparison principle implies \eqref{eq:upper-cigar-barrier} as desired.

Finally, \eqref{eq:lower-cigar-barrier} follows from the
fact that $\bigl(u(t)\bigr)_{t\in[0,\infty)}$ is maximally stretched and
we have at the time $t_1$ for all $s\in\mathbb R$ from \eqref{eq:refined-cigars-claim-2}
\begin{align*}
u(t_1,s) &\ge \begin{cases}
\log r_c & \text{for all $s\in(-\infty,s_1]$}\\
\log r_c - (s-s_1) & \text{for all $s\in(s_1,\infty)$}
\end{cases} \\
&\ge
\mathcal{C}_{r_c^{-2}}\bigl( s-s_1\bigr)
= \mathcal{C}_{r_c^{-2}}\bigl(0, s-s_1\bigr)
\qquad\text{using \eqref{eq:cigar-estimate};}
\end{align*}
hence $u(t,s)\ge
\mathcal{C}_{r_c^{-2}}\bigl(t-t_1,
s-s_1\bigr)\ge\mathcal{C}_{r_c^{-2}}\bigl(t-t_1, s\bigr) $ for
all $t\ge t_1$.
\end{proof}

\subsection{Bounding the metric at later times}
\label{barrier_final_sect}

The cigar barriers we have just constructed give good control on the
\textsc{cb}-Ricci flow after the bulb has deflated and while the cylinder part is shrinking. To regain control on the curvature afterwards, we will need the following:

\begin{lemma}[\textsf{Under cusp}]
\label{lemma:cusp-barrier}
For some parameters $r_c\in\left(0,\frac1{10}\right)$ and
$l_c= \frac18r_c^{-1}$ if we consider the \textsc{cb}-Ricci flow
$\bigl(g(t)\bigr)_{t\in[0,\infty)}$
on $\mathbb C$,
then
\begin{equation}
\label{eq:uniform-cusp-barrier}
g(\nicefrac72) \le 256\, g_{\mathrm{hyp}} \quad\text{on $\Disc\setminus\{0\}$,}
\end{equation}
where $g_{\mathrm{hyp}}$ is the complete hyperbolic metric on
$\Disc\setminus\{0\}$.
\end{lemma}
\begin{proof}
As usual we write $g(t)=\ee^{2u(t,s)}\bigl(\ds^2+\df\theta^2\bigr)$,
and fix $t_1\in\bigl(\nicefrac34,\nicefrac52\bigr)$ from Lemma
\ref{lemma:refined-cigar-barriers}. Recall from Lemma
\ref{lemma:cb-prop}(ii), we have
$s_e=s_0+\frac12\log\left(1+r_c^2\right)\in\left(-\frac{l_c+\nicefrac\pi2}{r_c},-\frac{l_c}{r_c}\right)$ and
$\Disc\setminus\{0\}\simeq(s_e,\infty)\times[0,2\pi)$, in particular
$g_{\mathrm{hyp}}=\frac{\ds^2+\df\theta^2}{(s-s_e)^2}$.

At the time $t=t_1+1$ we have on the one hand side for
$s\in\left(s_e,-\frac{l_c-2}{r_c}\right]$ by Corollary
\ref{cor:chen-lower-curv-bd-cf} and the construction of the metric
$g_{\mathsf{cb}}$
\begin{align*}
u(t_1+1,s) &\le \frac12\log\bigl(2(t_1+1)+1\bigr) + u_{\mathsf{cb}}(s) \\
&\le \frac12\log8 +
\begin{cases}
-\log\Bigl(r_c^{-1}\cos(r_cs+l_c)\Bigr) & \text{for
$s\in\left(s_e,-\frac{l_c}{r_c}\right)$}\\
\log r_c & \text{for $s\in\left[-\frac{l_c}{r_c},-\frac{l_c-2}{r_c}\right]$}
\end{cases}\\
&\le \frac12\log8 -\log\left(s+\frac{l_c+\frac\pi2}{r_c}\right)
+
\begin{cases}
\log\frac\pi2 &\text{for
$s\in\left(s_e,-\frac{l_c}{r_c}\right)$}\\
\log\left(2+\frac\pi2\right) & \text{for
$s\in\left[-\frac{l_c}{r_c},-\frac{l_c-2}{r_c}\right]$}
\end{cases}\\
&\le -\log\left(s-s_e\right)  + \log15 \qquad\textstyle\text{for
all $s\in\left(s_e,-\frac{l_c-2}{r_c}\right]$.}
\end{align*}
On the other hand, by virtue of Lemma
\ref{lemma:refined-cigar-barriers} we know that $u(t_1+1,s)\le
\mathcal{C}_{(4r_c)^{-2}}\bigl(1,s-2r_c^{-1}\bigr)$
for all $s\ge-\frac{l_c-2}{r_c}$, so it suffices to compare the cusp
of $g_{\mathrm{hyp}}$ with this cigar in this part.
Note that for all $s\ge-\frac{l_c-2}{r_c}=-\frac18r_c^{-2}+2r_c^{-1}$ we have
\begin{align*}
\frac{\df}{\df s}\Bigl(
\mathcal{C}_{(4r_c)^{-2}}\bigl(1,s-2r_c^{-1}\bigr)\Bigr)
&= \frac{\df}{\df s} \Bigl(
\mathcal{C}\bigl(\textstyle\frac18r_c^{-2}+s-2r_c^{-1}\bigr)\Bigr) \\
&\le 
\mathcal{C}'(0)
=  -\frac12
\le -\frac1{s+\frac{l_c}{r_c}} \le
-\frac1{s-s_e}
= \frac{\df}{\df s}\Bigl(
-\log\left(s-s_e\right) \Bigr).
\end{align*}
Combined with the following inequality at $s=-\frac{l_c-2}{r_c}$
\begin{align*}
\mathcal{C}_{(4r_c)^{-2}}\left(1,-\frac{l_c}{r_c}\right)
&\le \log\left(4r_c\right) &&\text{using
\eqref{eq:cigar-estimate}}\\
&= \log4-\log\left(-\frac{l_c-2}{r_c}+\frac{l_c+\frac\pi2}{r_c}\right)
+ \log\left(2+\frac\pi2\right)\\
&\le
-\log\left(-\frac{l_c-2}{r_c}-s_e\right)
+ \log15,
\end{align*}
we may conclude that
$$\mathcal{C}_{(4r_c)^{-2}}\bigl(1,s-2r_c^{-1}\bigr)
\le -\log(s-s_e)+\log 15\qquad \text{for }s\ge -\frac{l_c-2}{r_c}$$
and thus
$u(t_1+1,s) \le
-\log\left(s-s_0\right)+\log15$ for all $s\ge s_e$, or equivalently
$g(t_1+1) \le 225\,g_{\mathrm{hyp}}$ on
$\Disc\setminus\{0\}$. Consequently, we have
\[ g(t) \le \bigl(2(t-t_1)+225\bigr)\,g_{\mathrm{hyp}} \qquad\text{on
$\Disc\setminus\{0\}$ for all $t\in[t_1,\infty)$} \]
because the right-hand side is a maximally stretched Ricci flow on
$\Disc\setminus\{0\}$; in particular, we have
\eqref{eq:uniform-cusp-barrier}.
\end{proof}

Following \cite[Lemma 2.9]{Top11} we use the following result from
\cite{Top12} in order to bound uniformly the conformal factor.

\begin{lemma}[Special case of {\cite[Lemma 3.3]{Top12}}]
\label{lemma:cf-bd-hyp-cusp}
If $\bigl(\ee^{2v(t)}|\dz|^2\bigr)_{t\in[0,1]}$ is any smooth Ricci
flow on $\Disc$  with $\ee^{2v(0)}|\dz|^2\le g_{\mathrm{hyp}}$ on
$\Disc\setminus\{0\}$, then there exists a universal constant
$\beta\in(0,\infty)$ such that
\begin{equation}
\label{eq:rev-cusp-key-lemma}
\sup_{\Disc_{\frac12}} v(t) \le \frac\beta t
\qquad\text{for all $t\in(0,1]$.}
\end{equation}
\end{lemma}

\begin{lemma}[{Variant of \cite[Lemma 2.9]{Top11}}]
\label{lemma:plane_bounds}
There exists a universal constant $C\in(1,\infty)$ such that if we
consider the \textsc{cb}-Ricci flow
$\bigl(g(t)\bigr)_{t\in[0,\infty)}$ on $\mathbb C$
with
$r_c\in\left(0,\frac1{10}\right)$ and $l_c=\frac18r_c^{-1}$, then
\begin{equation}
\label{eq:uniform-upper-bound-disc}
g(t) \le C |\dz|^2 \qquad\text{for
all $t\in\left[\nicefrac{15}4,\infty\right)$.}
\end{equation}
\end{lemma}
\begin{proof}
Writing $g(t)=\ee^{2u(t)}|\dz|^2$, we define the parabolically
rescaled (and translated) Ricci flow
\[ \ee^{2v(t)}|\dz|^2 := \frac1{256}\,g(\nicefrac72+256t). \]
By Lemma \ref{lemma:cusp-barrier} we have
$\ee^{2v(0)}|\dz|^2=\frac1{256}g(\nicefrac72) \le g_{\mathrm{hyp}}$ on
$\Disc\setminus\{0\}$, satisfying the hypothesis of Lemma
\ref{lemma:cf-bd-hyp-cusp}, and we may conclude
\[ g\bigl(\nicefrac{15}4\bigr) = 256\, \ee^{2v(\nicefrac1{1024})}|\dz|^2
\le 256\,\ee^{2048\beta}\,|\dz|^2 \quad\text{on $\Disc_{\frac12}$.} \]
For the upper bound outside of $\Disc_{\frac12}$, first note that in
logarithmic cylindrical coordinates we have
$\Disc_{\frac12}\setminus\{0\}\simeq(s_e+\log2,\infty)\times[0,2\pi)\supset
\overline{\mathcal{U}_c\cup\mathcal{U}_b}$, hence
$u_{\mathsf{cb}}(s_e+\log2)=u_2(s_e+\log2)$ using the notation
from Definition \ref{defn:cb-metric}. Since
$u'_{\mathsf{cb}}(s)\ge-1$ for all $s\in(-\infty,s_e+\log2]$, it
suffices to estimate it at $s=s_e+\log2$
\begin{align*}
u_{\mathsf{cb}}(s_e+\log2) &=
-\log\left(r_c^{-1}\cos\bigl(r_c(s_e+\log2)+l_c\bigr)\right) \\
&=
-\log\Bigl(r_c^{-1}\cos\bigl(r_c(\log2+\textstyle\frac12\log(1+r_c^2))-\arctan
r_c^{-1}\bigr)\Bigr)  \\
&\le -\log\left(r_c^{-1}\cos\arctan r_c^{-1}\right) =
\frac12\log\bigl(1+r_c^2\bigr) \\
&\le - (s_e+\log2) +s_e + \frac32\log2,
\end{align*}
i.e. $u_{\mathsf{cb}}(s)\le - s +s_e + \frac32\log2$ for all
$s\le s_e+\log 2$.
Using Corollary \ref{cor:chen-lower-curv-bd-cf}, we get
\[ u\bigl(\nicefrac{15}4,s\bigr) \le u_{\mathsf{cb}}(s) +
\frac12\log\left(\frac{15}2+1\right)  \le -s + s_e
+\frac12\log68 \qquad\text{for all $s\in(-\infty,s_e+\log2]$,} \]
or $g\bigl(\nicefrac{15}4\bigr) \le 68 |\dz|^2$ on $\mathbb
C\setminus\Disc_{\frac12}$.
Choosing $C:=256\ee^{2048\beta}>68$, the static (and maximally
stretched) Ricci flow
$\bigl(C|\dz|^2\bigr)_{t\in\left[\nicefrac72,\infty\right)}$ is an
upper barrier for $g(t)$ from time $t=\nicefrac{15}4$ onwards.
\end{proof}

\section{Burst of large curvature}
\label{large_burst_sect}

At this point, we have derived the full set of estimates on the \textsc{cb}-Ricci flow, in particular to show that it will have a burst of large curvature on some time interval during the flow. However, the Ricci flow  required for Theorem \ref{thm:ex-rf-curv-bursts} will evolve from infinitely many copies of the \textsc{cb}-surface, and each of these will evolve as slight perturbations of the \textsc{cb}-Ricci flow. Therefore we need:

\begin{prop}[\textsf{Lower curvature bound}]
\label{prop:lower-curvature-bound}
For parameters $r_c\in(0,\frac1{256})$ and
$l_c\ge \frac18 r_c^{-1}$ let
$\bigl(g_{\mathsf{cb}}(t)\bigr)_{t\in[0,\infty)}$ be the
\textsc{cb}-Ricci flow on $\mathbb C$.
Suppose $\bigl(g(t)\bigr)_{t\in[0,3]}$ is a Ricci flow
on $\Disc_2\subset\mathbb C$ conformally equivalent to
$g_{\mathsf{cb}}(0)\bigr|_{\Disc_2}$ such that for
$\alpha\in\bigl(1,\nicefrac{201}{200}\bigr]$
\begin{equation}
\label{hyp}
\alpha^{-2} g_{\mathsf{cb}}\bigl(\alpha^2 t\bigr)\Bigr|_{\Disc_2}
\le g(t) \le \alpha^2 g_{\mathsf{cb}}\bigl(\alpha^{-2}
t\bigr)\Bigr|_{\Disc_2} \qquad\text{for all $t\in[0,3]$.}
\end{equation}
Then there exists a time $t_1\in\bigl(\frac34,\frac83\bigr)$ such that
we have the curvature estimate
\begin{align}
\label{eq:unbounded-curv}
\sup_{\Disc_2} \GK_{g(t)} &\ge \frac1{r_c} &&\text{for all
$t\in\left[t_1,t_1+\frac1{100}\right]$.}
\end{align}
\end{prop}
\begin{proof}
Let $\tilde t_1\in\bigl(\frac34,\frac52\bigr)$ be the  time when the
barriers from Lemma \ref{lemma:refined-cigar-barriers} begin to hold for
$\bigl(g_{\mathsf{cb}}(t)\bigr)_{t\in[0,\infty)}$.
Define $t_1:=\alpha^2\tilde t_1\in\left(\frac34,\frac83\right)$ and
\[ \rho(t) := \max\Bigl\{ r>0 : \Vol_{g(t)} \gBall_{g(t)}(0;r) \le
2\pi r_c \Bigr\}. \]
To see that
$\gBall_{g(t)}\bigl(0;\rho(t)\bigr)\Subset
\mathcal{U}_b\cup\overline{\mathcal{U}_c}$
for sufficiently long after $t=t_1$,
we can use the lower barrier \eqref{eq:lower-cigar-barrier}
and \eqref{eq:cigar-area-t} to estimate
\begin{align*}
\Vol_{g(t)}(\mathcal{U}_c\cup\mathcal{U}_b)
&\ge
\Vol_{\alpha^{-2}g_{\mathsf{cb}}(\alpha^2t)}(\mathcal{U}_c\cup\mathcal{U}_b)\\
&\ge \alpha^{-2}   \Vol_{\mathcal{C}_{r_c^{-2}}}
\left(\left[2r_c^{-2}\left(\alpha^2t-\tilde t_1\right)-\frac{l_c}{r_c},\infty\right)
\times[0,2\pi)\right)\\
&\ge 2\pi \left(\frac{r_cl_c}{\alpha^2} +\frac{2t_1}{\alpha^4}-2t
\right) \ge 2\pi\, r_c \\
& \hspace{10em}\text{for
$t\in\left[\alpha^{-4}t_1,\alpha^{-4}t_1+\frac{r_c}2\bigl(\alpha^{-2}l_c-1\bigr)\right]$.}
\end{align*}
Note that
\begin{align*}
\alpha^{-4}t_1 + \frac{r_c}2\left(\alpha^{-2}l_c-1\right)
&= t_1 + \left(\alpha^{-4}-1\right)t_1 +
\frac{r_c}2\left(\alpha^{-2}l_c-1\right)
\ge t_1+ \frac1{100},
\end{align*}
so $\left[t_1,t_1+\frac1{100}\right]\subset
\left[\alpha^{-4}t_1,\alpha^{-4}t_1+\frac{r_c}2\bigl(\alpha^{-2}l_c-1\bigr)\right]$.
The upper barrier \eqref{eq:upper-cigar-barrier}  and \eqref{eq:cigar-estimate} gives us an
estimate of the length of the boundary at $s=-\frac{l_c}{r_c}$
\[
\L_{g(t)}\partial
\gBall_{g(t)}\bigl(0;\rho(t)\bigr)
\le \L_{\alpha^2g_{\mathsf{cb}}\left(\alpha^{-2}t\right)}
\partial
\gBall_{g(t)}\bigl(0;\rho(t)\bigr)
\le \alpha\,8\pi\, r_c \]
for all $t\in\left[\alpha^2\tilde t_1,\alpha^2(\tilde t_1+1)\right]=
[t_1,t_1+\alpha^2]\supset\left[t_1,t_1+\frac1{100}\right]$.
Applying Bol's isoperimetric inequality (Theorem
\ref{thm:bol-isop-ineq}) to $\gBall_{g(t)}\bigl(0;\rho(t)\bigr)$ yields
\begin{align*}
\sup_{\Disc_2} \GK_{g(t)} &\ge
\sup_{\gBall_{g(t)}\left(0;\rho(t)\right)} \GK_{g(t)} \\
&\ge \frac{4\pi}{\Vol_{g(t)} \gBall_{g(t)}\bigl(0;\rho(t)\bigr)}
- \left(
\frac{\L_{g(t)}\partial\gBall_{g(t)}\bigl(0;\rho(t)\bigr)}{\Vol_{g(t)}
\gBall_{g(t)}\bigl(0;\rho(t)\bigr)}\right)^2 \\
&\ge 2r_c^{-1} - 16\alpha^2 \ge 2r_c^{-1}-17\ge r_c^{-1}
\end{align*}
for all $t\in\left[t_1,t_1+\frac1{100}\right]$.
\end{proof}

To apply this proposition, we need a way to check that the Ricci flows that we construct do satisfy the hypothesis \eqref{hyp}.

\begin{lemma}
\label{lemma:isolated-cb-metric}
For any $r_c\in(0,1)$ let $\bigl(g_{\mathsf{cb}}(t)\bigr)_{t\in[0,\infty)}$ be a
\textsc{cb}-Ricci flow on $\mathbb C$. For all $\alpha>1$ and $T>0$, there exists $R>2$ such that if $\bigl(g(t)\bigr)_{t\in[0,T]}$ is a complete
Ricci flow on any surface $\m$ such that $\bigl(\m,g(0)\bigr)$
contains an isometric copy of $\bigl(\Disc_R,g_{\mathsf{cb}}(0)\bigr)$
then we have
\[ \alpha^{-2} g_{\mathsf{cb}}\bigl(\alpha^2 t\bigr)\Bigr|_{\Disc_2}
\le g(t) \le \alpha^2 g_{\mathsf{cb}}\bigl(\alpha^{-2}
t\bigr)\Bigr|_{\Disc_2} \qquad\text{for all $t\in[0,T]$.} \]
\end{lemma}
\begin{proof}
First note that from Lemma \ref{lemma:cb-prop}(ii) we have
$g_{\mathsf{cb}}(0)\bigr|_{\mathbb
C\setminus\Disc_2}=|\dz|^2$. For $v_0=\pi$ let $B>0$ be the
constant from Theorem \ref{thm:chen-2d-local-curv-estim}. With
\begin{equation}
\label{eq:defn_r0}
r_0=\max\left\{\sqrt{B\alpha^2T},\sqrt{\frac{4\alpha^2T}{\log\alpha}} \right\}
\end{equation}
we choose $R=2r_0+2$ and have the isometry
\[ \psi : \bigl(\Disc_R,g_{\mathsf{cb}}(0)\bigr) \to
\bigl(\mathcal U_R,g(0)\bigr)\subset\bigl(\m,g(0)\bigr). \]
We are going to apply Theorem \ref{thm:chen-2d-local-curv-estim} to
$\bigl(g_{\mathsf{cb}}(t)\bigr)_{t\in[0,\alpha^2 T]}$ and
$\bigl(g(t)\bigr)_{t\in[0,T]}$
at points $p\in\partial\Disc_{r_0+2}$ and
$q\in\psi\bigl(\partial\Disc_{r_0+2}\bigr)\subset\mathcal
U_R$. With the above choices of $p$, $q$, $r_0$, $g_{\mathsf{cb}}(0)$
and $g(0)$ the Ricci flows $\bigl(g_{\mathsf{cb}}(t)\bigr)_{t\in[0,\alpha^2 T]}$ and
$\bigl(g(t)\bigr)_{t\in[0,T]}$  satisfy conditions (i),(ii) and
(iii) of Theorem \ref{thm:chen-2d-local-curv-estim}, such that
we may conclude
\begin{align*}
\Bigl|\GK_{g_{\mathsf{cb}}(t)}(p)\Bigr|&\le 2r_0^{-2} 
&&\text{for all $p\in\partial\Disc_{r_0+2}$ and $t\in\left[0,\alpha^2T\right]$}  \\
\qquad\text{and}\qquad \Bigl|\GK_{g(t)}(q)\Bigr|&\le 2r_0^{-2}
&&\text{for all $q\in\psi\bigl(\partial\Disc_{r_0+2}\bigr)$ and
$t\in[0,T]$.}  
\end{align*}
Integrating the Ricci flow equation $\pddt g(t)=-2\GK_{g(t)}g(t)$
along with this curvature estimate, we obtain using
\eqref{eq:defn_r0}, (i.e. $\ee^{4r_0^{-2}t}\le\alpha$
for all $t\in\left[0,\alpha^2T\right]$)
\begin{align}
\alpha^{-1}\,g_{\mathsf{cb}}(0) \le \ee^{-4r_0^{-2}t}g_{\mathsf{cb}}(0)
&\le g_{\mathsf{cb}}(t) \le \ee^{4r_0^{-2}t} g_{\mathsf{cb}}(0) \le
\alpha\, g_{\mathsf{cb}}(0) 
\end{align}
on $\partial\Disc_{r_0+2}$ for all $t\in\left[0,\alpha^2T\right]$ and 
\begin{align}
\alpha^{-1}\,g(0)\le \ee^{-4r_0^{-2}t}g(0) &\le g(t) \le \ee^{4r_0^{-2}t} g(0) \le \alpha\, g(0)
\end{align}
on $\psi\bigl(\partial\Disc_{r_0+2}\bigr)$ for all $t\in[0,T]$.
Hence we may estimate on $\partial\Disc_{r_0+2}$
\begin{align*}
\alpha^{-2}g_{\mathsf{cb}}\bigl(\alpha^2t\bigr) &\le \alpha^{-1}g_{\mathsf{cb}}(0) =
\alpha^{-1}\psi^*g(0)
\le \psi^*g(t) \le \alpha \psi^*g(0) = \alpha g_{\mathsf{cb}}(0)
\le \alpha^2 g_{\mathsf{cb}}\bigl(\alpha^{-2}t\bigr)
\end{align*}
for all $t\in[0,T]$.
Therefore, we may apply a comparison principle and conclude the
lemma's statement.
\end{proof}

\section{Burst of unbounded curvature}
\label{sec:burst-unb-curv}

\begin{proof}[Proof of \textbf{Theorem \ref{thm:ex-rf-curv-bursts}}]
We begin by proving the last part of the theorem, i.e. assuming that the underlying Riemann surface is $\mathbb C$, and proving 
\eqref{mainclaim1}.

For $T=3$ and $\alpha=\nicefrac{201}{200}$ we obtain from Lemma
\ref{lemma:isolated-cb-metric} some radius $R>2$.
Pick any sequence of disjoint discs
$\bigl(\Ball(p_j;R)\bigr)_{j\in\mathbb N}$ in $\C$, of radius $R$. 
We obtain the metric $g_0$ on $\mathbb C$ by replacing those discs with \textsc{cb}-metrics 
$g_{\mathsf{cb}}\bigr|_{\Disc_R}$
with cylinder radius $r_c=\frac1{256j}$ and length 
$l_c=\frac18r_c^{-1}=32j$ for each $j\in\mathbb N$.
Now observe that by construction (Lemma \ref{lemma:cb-prop}) we
have both $g_0\ge|\dz|^2$ and $\GK_{g_0}\in\left[-1,\frac12\right]$.
Then Theorem \ref{thm:2d-existence} provides a complete (and
maximally stretched) Ricci flow $\bigl(g(t)\bigr)_{t\in[0,\infty)}$
starting with $g(0)=g_0$. For a short time this Ricci flow coincides
with the Shi solution which has bounded curvature, so we can apply
a maximum principle to the evolution equation of the Gaussian curvature
$\pddt\GK_{g(t)}=\Delta_{g(t)} \GK_{g(t)}+ 2\GK_{g(t)}^2$ 
(see \cite[Proposition 2.5.4]{Top06})
and conclude
\begin{equation}
\label{eq:bdd-curv-1}
\GK_{g(t)} \le \frac1{2(1-t)} \qquad \text{for all $t\in[0,1)$}. 
\end{equation}
By virtue of Lemma \ref{lemma:isolated-cb-metric} we may apply
Proposition \ref{prop:lower-curvature-bound} to
$\bigl(g(t)\bigr)_{t\in[0,3]}$ on each disc $\Ball(p_j;2)$ and
obtain a sequence of times $\bigl(t_{1,j}\bigr)_{j\in\mathbb
N}\in\left(\frac34,\frac83\right)$ such that for each $j\in\mathbb
N$
\[ \sup_{\Ball(p_j;2)} \GK_{g(t)} \ge 256j \qquad\text{for all
$t\in\left[t_{1,j},t_{1,j}+\frac1{100}\right]$.} \]
Therefore, there exists a $t_1\in\left[\frac34,\frac83\right]$ such that
for every $j_0\in\mathbb N$ and
$t\in\left(t_1,t_1+\frac1{100}\right)$ we can find $j\ge j_0$ with
$t\in\left[t_{1,j},t_{1,j}+\frac1{100}\right]$. 
Consequently, we
have
\[ \sup_{\mathbb C} \GK_{g(t)} = \infty \qquad \text{for all
$t\in\left(t_1,t_1+\frac1{100}\right)$.} \]
By \eqref{eq:bdd-curv-1} we know that $t_1\ge1$, hence we may assume
that $t_1\in[1,3)$. 

Let $C>1$ be the universal constant from Lemma
\ref{lemma:plane_bounds}. Using again Lemma
\ref{lemma:isolated-cb-metric} we obtain
\[ g\left(\frac{15}4\alpha^2\right) \le \alpha^2 C |\dz|^2 \quad\text{on
$\Ball(p_j;2)$ for all $j\in\mathbb N$.} \]
On the complement we can simply use Corollary
\ref{cor:chen-lower-curv-bd-cf} in order to conclude
\[ g\left(\frac{15}4\alpha^2\right) \le
\left(\frac{15}2\alpha^2+1\right)|\dz|^2 \quad\text{on $\mathbb C 
\setminus\bigcup\limits_{j\in\mathbb N}\Ball(p_j;2)$.} \]
Assuming (without any restriction) $C>\frac{15}2\alpha^2+1$, we have
$g\left(\frac{15}4\alpha^2\right)\le \alpha^2C\, |\dz|^2$ on $\mathbb C$, and
consequently $g(t)\le \alpha^2C\,|\dz|^2$ for all
$t\in\left[\frac{15}4\alpha^2,\infty\right)$, because the right-hand
side is a (static) maximally stretched Ricci flow on $\mathbb C$. On
the other hand we have $g(t)\ge|\dz|^2$ for all $t\in[0,\infty)$,
because $\bigl(g(t)\bigr)_{t\in[0,\infty)}$ is maximally stretched too.
With the solution sandwiched between $|\dz|^2$ and $\alpha^2C\,|\dz|^2$, 
parabolic regularity shows that we have uniform $\mathrm C^k$
bounds for $t\in[4,\infty)$; in particular, the curvature is bounded
for all $t\ge4$. 
This completes the proof in the case of $\mathbb C$.

Consider now the case that the underlying space is $\mathbb D$, and let $H$ be the Poincar\'e metric thereon. As before, we set $T=3$ and $\al=\nicefrac{201}{200}$, and obtain $R>2$ from Lemma \ref{lemma:isolated-cb-metric}.

\begin{figure}[th]
  \centering
  \includegraphics{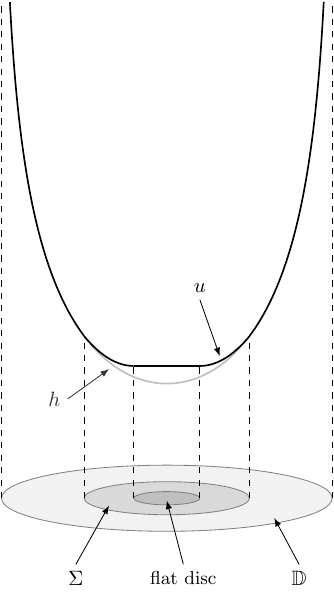}
  \caption{Conformal factors for metrics $H=\ee^{2h}|\dz|^2$ and $g_1=\ee^{2u}|\dz|^2$ on the disc $\mathbb D$}
  \label{fig:conf-disc}
\end{figure}

We begin by `flattening out' $H$ in a small neighbourhood of some point in $\mathbb D$, to give a new metric $g_1$. More precisely, we make a smooth conformal perturbation of $H$ in some domain $\Sigma\subset\subset \mathbb D$ so that the new surface $\bigl(\mathbb D, g_1\bigr)$ contains some (say small) flat disc, while only increasing the metric and retaining the nonpositivity of the curvature, and call the resulting metric $g_1$. (Think of flattening the conformal factor in a neighbourhood of the `origin' $\to$ Figure \ref{fig:conf-disc}.)

We now do the same perturbation near a sequence of points heading out to infinity in $\mathbb D$, with the perturbed regions being disjoint, to give a new metric $g_2$. 
More precisely, pick any isometry $\Gamma$ of $\bigl(\mathbb D, H\bigr)$ such that the images $\Gamma^k(\Sigma)$ are pairwise disjoint for $k\in \N$ (for example one can take a sufficiently large translation in the upper half-space model of the hyperbolic plane)
and define $g_2$ at each point to be the supremum of $(\Gamma^k)^*g_1$ as $k$ varies within $\N$.
The resulting metric $g_2$ is bounded below by $H$, and bounded above by some multiple of $H$. 

Now homothetically expand $g_2$ by a large enough factor so that the infinitely many flat discs that lie isometrically within $\bigl(\mathbb D, g_2\bigr)$ end up of radius at least $R$, and so that the curvature lies within $[-1,0]$. Still, the metric is sandwiched from below by $H$ and from above by some large multiple of $H$. We call this expanded metric $g_3$, and one can think of it as the substitute for the metric $|\dz|^2$ in the $\mathbb C$ case, into which we inserted \textsc{cb}-metrics. 

We now replace each of the flat discs of radius $R$ by a sequence of more and more extreme \textsc{cb}-metrics as we did in the 
$\mathbb C$ case, to get the metric $g_0$ that we are going to flow to obtain the Ricci flow required in the theorem.
By construction, we have $g_0\geq g_3\geq H$, although $g_0$ is not bounded above by any multiple of $H$.

Note that $g_0$ has curvature lying within $[-1,\nicefrac12]$, so as in the case for $\C$, the curvature will initially be bounded as dictated by \eqref{eq:bdd-curv-1}, and just as before, we then have a time interval $(t_1,t_1+\nicefrac1{100})$ when the curvature is unbounded. Moreover, at time $t=\frac{15}4 \alpha^2$ we may conclude that the flow $g(t)$ lies below the metric $\alpha^2 C\, g_3$, in a manner identical to how we dealt with the $\mathbb C$ case.

Thus, at this time, the flow is sandwiched between $H$ and some large
multiple of $H$. We may then appeal to Theorem \ref{geom_thm}, coming
from \cite{GT11}, applied to the Ricci flow $g\!\left(t+\frac{15}4 \alpha^2\right)$, to deduce that $\frac1{2t}\,g(t)$ must converge in $\Cts^k$ to $H$ as $t\to\infty$.
\end{proof}

\begin{proof}[Sketch of proof of \textbf{Theorem \ref{thm:ex-rf-delayed-unbounded-curv}}]
Let us begin by considering the case that the underlying Riemann surface is $\mathbb C$, since in that case it is easiest to directly apply the precise estimates we have derived to prove
Theorem \ref{thm:ex-rf-curv-bursts}, even though they give much more than we now require. Indeed, in this case, the proof is similar to the proof of Theorem \ref{thm:ex-rf-curv-bursts} on $\mathbb C$, but in the construction we choose much longer cylinders, e.g. $l_c=r_c^{-2}$. Combined with the
techniques in \cite{GT13} (in particular using pseudolocality in
the cylinder region, or alternatively barrier arguments), we could modify Lemma
\ref{lemma:refined-cigar-barriers} such that the upper barrier
\eqref{eq:upper-cigar-barrier} holds for longer times. An inspection
of the proof of Proposition \ref{prop:lower-curvature-bound} shows
that we also have the lower curvature bound
\eqref{eq:unbounded-curv} for longer and longer times, which is enough to conclude the result in this case.

\begin{figure}[th]
  \centering
  \includegraphics{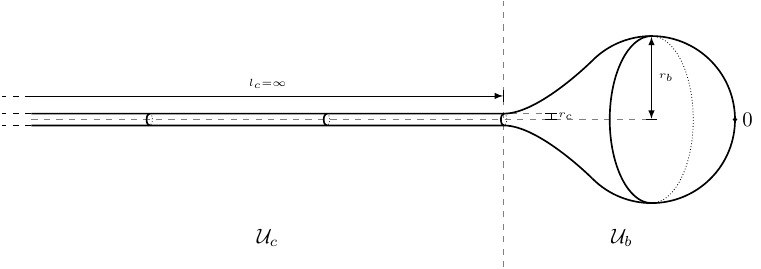}
  \caption{Long `lollipop' surface: Infinite cylinder with bulb cap}
  \label{fig:lollipop2}
\end{figure}
In the case that we work on a more general underlying Riemann surface $\Mf$, we take a slightly different approach, which is a lot simpler than that of Theorem \ref{thm:ex-rf-curv-bursts} for not having to worry about making the curvature bounded for large times. Consider instead of the \textsc{cb}-surface we describe in Section \ref{strategy_sect}, just the cylinder $\mathcal{U}_c$ and bulb  $\mathcal{U}_b$ parts, but with the cylinder infinitely long ($\to$ Figure \ref{fig:lollipop2}). This `lollipop' surface has a subsequent Ricci flow whose behaviour should be quite apparent by now: The curvature, although initially bounded above and below, will gradually blow up as the bulb part shrinks, until it reaches a large value depending on $r_c$, from which moment the Ricci flow will look more and more like a thin cigar, with the large curvature persisting. This can be made precise using slight variants of the estimates we proved in Sections \ref{width_sect} and \ref{barrier_sect} (without requiring Section \ref{barrier_final_sect}) and in Proposition \ref{prop:lower-curvature-bound}. 
As in the $\mathbb C$ case of this sketch proof, 
the variant of Lemma \ref{lemma:isolated-cb-metric} will have a slightly adjusted proof using the techniques from \cite{GT13} or barrier arguments.


Equipped with this long lollipop surface, we construct the initial metric $g_0$ required to prove Theorem \ref{thm:ex-rf-delayed-unbounded-curv} on $\Mf$ as follows. Pick a sequence of points $\left(p_j\right)_{j\in\mathbb N}$ in $\Mf$ without any accumulation point (i.e. heading off to infinity) and let $H$ be the complete conformal hyperbolic metric on $\Mf\setminus\bigcup_{j\in\mathbb N} \{p_j\}$. Near each of the punctures $p_j$, the metric $H$ will have the structure of a hyperbolic cusp, which we can truncate further and further out as $j$ increases, and add in a lollipop surface with smaller and smaller $r_c$, that has itself had its cylinder part truncated. 
Providing we truncate the cylinder to have length $l_c$, where $l_c r_c\to\infty$ as $j\to\infty$, our new surface will generate a Ricci flow that has the properties demanded by Theorem \ref{thm:ex-rf-delayed-unbounded-curv}.
\end{proof}


\begin{appendix}
\section{\textsl{A priori} estimates}
\begin{thm}[\textsc{Chen} {\cite[Corollary 2.3(i)]{Che09}}]
\label{thm:chen-lower-curv-bd}
Let $\bigl(g(t)\bigr)_{t\in[0,T]}$ be a complete Ricci flow on a
surface $\Mf^2$.
If $\GK_{g(0)}\ge -K_0$ for some $K_0\in[0,\infty]$, then
\begin{equation}
\label{eq:chen-lower-curv-bd}
\GK_{g(t)} \ge -\frac1{2t+K_0^{-1}}
\qquad\text{for all }t\in[0,T].
\end{equation}
\end{thm}

\begin{cor}
\label{cor:chen-lower-curv-bd-cf}
Let $\bigl(g(t)\bigr)_{t\in[0,T]}$ be a complete
Ricci flow on a surface $\Mf^2$, such that $\GK_{g(0)}\ge-K_0$ for
some $K_0\in[0,\infty)$. Then writing in a local complex
isothermal coordinate $g(t)=\ee^{2u(t)}|\dz|^2$, we have
\begin{equation}
\label{eq:chen-lower-curv-bd-cf}
u(t_2) \le u(t_1) + \frac12\log\frac{2t_2+K_0^{-1}}{2t_1+K_0^{-1}}
\qquad\text{for all $0\le t_1<t_2\le T$.}
\end{equation}
\end{cor}

A more elaborate argument of \textsc{Chen} leads to the following
pseudolocality-type result giving two-sided estimates on the curvature.
\begin{thm}[\textsc{Chen} {\cite[Proposition 3.9]{Che09}}]
\label{thm:chen-2d-local-curv-estim}
Let $\bigl(g(t)\bigr)_{t\in[0,T]}$ be a Ricci flow on a surface
$\Mf^2$. If we have for some $p\in\Mf$, $r_0>0$ and $v_0>0$
\begin{compactenum}[(i)]\medskip
\item $\gBall_{g(t)}(p;r_0)\Subset\Mf$ for all $t\in[0,T]$;
\item $\Bigl| \GK_{g(0)}\Bigr| \le r_0^{-2}$ on $\gBall_{g(0)}(p;r_0)$;
\item $\Vol_{g(0)}\gBall_{g(0)}(p;r_0) \ge v_0 r_0^2$,
\end{compactenum}\medskip
then there exists a constant $B=B(v_0)>0$ such that for all
$t\in\bigl[0,\min\bigl\{T,\frac1B r_0^2\bigr\}\bigr]$
\[ \Bigl| \GK_{g(t)} \Bigr| \le 2r_0^{-2}
\qquad\text{on } \gBall_{g(t)}\Bigl(p;\frac{r_0}2\Bigr). \]
\end{thm}

The following isoperimetric inequality due to G. Bol allows us
estimate the maximum of the curvature on a surface's domain from
below if we know its area and the length of its boundary. For an
alternative proof using curvature flows,
and further generalisations see
\cite{Top98} and \cite{Top99}.

\begin{thm}{\rm\cite[eqn.~(30) on p.~230]{Bol41}}
\label{thm:bol-isop-ineq}
Let $\Omega$ be a simply-connected domain on a surface
$\bigl(\Mf^2,g\bigr)$, then
\begin{equation}
\label{eq:bol-isop-ineq}
\bigl(\mathrm{L}_g\,\partial\Omega\bigr)^2 \ge 4\pi \Vol_g(\Omega)
- \bigl(\Vol_g\Omega\bigr)^2 \sup_\Omega \GK_g.
\end{equation}
\end{thm}

\end{appendix}

\bibliography{2d_ricci_flow}

\end{document}